\documentclass[reqno,10pt]{amsart}

\usepackage{amsmath}
\usepackage{amssymb}
\usepackage{amsfonts}
\usepackage{graphicx}
\usepackage{enumerate}
\usepackage{dsfont}
\usepackage{color}

\newcommand{\R}{\mathds{R}}                   
\newcommand{\z}{\mathds{Z}}
\newcommand{\q}{\mathds{Q}}
\newcommand{\CP}{\mathds{C}\mathrm{P}}

\newcommand{\C}{\mathds{C}}

\newcommand{\K}{K\"{a}hler}
\newcommand{\GW}{GW}

\newcommand{\mult}{{\operatorname{mult}}}

\usepackage{setspace}
\newtheorem{theor}{Theorem}

\newtheorem{lem}[theor]{Lemma}
\newtheorem{cor}[theor]{Corollary}
\newtheorem{remark}[theor]{Remark}

\begin{document}

\title[Some remarks on the Gromov width of homogeneous Hodge manifolds]{Some remarks on the Gromov width of homogeneous Hodge manifolds
}

\author{Andrea Loi}
\address{(Andrea Loi) Dipartimento di Matematica 
         Universit\`a di Cagliari (Italy)}
         \email{loi@unica.it}

\author{Roberto Mossa}
\address{(Roberto Mossa) Dipartimento di Matematica 
         Universit\`a di Cagliari (Italy)}
         \email{roberto.mossa@gmail.com}

\author{Fabio Zuddas}
\address{(Fabio Zuddas) Dipartimento di Matematica e Informatica 
         Udine (Italy)}
\email{fabio.zuddas@uniud.it}

%
%

\begin{abstract}
We provide an upper bound for the Gromov width of compact homogeneous Hodge manifolds $(M, \omega)$
with $b_2(M)=1$. As an application we obtain an upper bound on the Seshadri constant $\epsilon (L)$
where $L$ is the ample line bundle on $M$ such that $c_1(L)=[\frac{\omega}{\pi}]$.
\end{abstract}



\maketitle

\section{Introduction}	
Consider the open ball of radius $r$,
\begin{equation}\label{ball}
B^{2n}(r)=\{(x, y)\in\R^{2n}\  |\  \sum_{j=1}^n x_j^2 + y_j^2<r^2 \}
\end{equation}
in the standard symplectic space $(\R^{2n}, \omega_0)$, where $\omega_0=\sum_{j=1}^n dx_j\wedge dy_j$.
The Gromov width of a $2n$-dimensional symplectic manifold $(M, \omega)$, introduced in \cite{GROMOV85}, is defined as
\begin{equation}\label{gromovwidth}
c_G(M, \omega)= \sup \{\pi r^2 \ |\ B^{2n}(r)\    \mbox{symplectically embeds into}  \  (M, \omega)\}.
\end{equation}

By Darboux's theorem $c_G(M, \omega)$ is a positive number.
In the last twenty years computations and estimates of the Gromov width for various examples have been
obtained by several authors (see, e.g.  \cite{LMZ13} and reference therein).

Gromov's width is an example of  \emph{symplectic  capacity} introduced  in \cite{HOFERZEHNDER90} (see also \cite{HOFERZEHNDER94}).
A map $c$ from the class  ${\mathcal C} (2n)$ of all symplectic manifolds of dimension $2n$ to $[0, +\infty]$
is called a \emph{symplectic capacity} if it satisfies the following conditions:

({\bf monotonicity}) if there exists a symplectic embedding $(M_1, \omega_1)\rightarrow (M_2, \omega_2)$ then 
$c(M_1, \omega_1)\leq c(M_2, \omega_2)$; 

({\bf conformality}) $c(M, \lambda\omega)=|\lambda|c(M, \omega)$, for every $\lambda\in\R\setminus \{0\}$; 

({\bf nontriviality}) $c(B^{2n}(r), \omega_0)=\pi r^2 =c(Z^{2n}(r), \omega_0)$.

\noindent
Here  $Z^{2n}(r)$ is the unitary   open cylinder in the standard $(\R^{2n}, \omega_0)$, i.e.
\begin{equation}\label{Zcil}
Z^{2n}(r)=\{(x, y)\in\R^{2n} \ | \ x_1^2+y_1^2<r^2\}.
\end{equation}

Note that the monotonicity property implies that $c$ is a symplectic invariant.
The existence of a capacity is not a trivial matter. 
It is easily seen that the Gromov width is the smallest symplectic capacity, i.e.   $c_G(M, \omega)\leq c (M, \omega)$ for any capacity $c$.  
Note that the  nontriviality property for $c_G$ comes from  the celebrated 
\emph{Gromov's nonsqueezing theorem} according to which  the existence of a symplectic  embedding of
$B^{2n}(r)$ into 
$Z^{2n}(R)$ implies $r\leq R$. Actually it is easily seen that the existence of any capacity implies 
Gromov's  nonsqueezing theorem. 

Recently  (see \cite{LMZ13}) the authors of the present paper have computed the Gromov width 
of all Hermitian symmetric spaces of compact and noncompact type and their products extending the
previous results of G. Lu \cite{LU06} (see also \cite{GWgrass})  for the case of complex Grassmanians.

The aim of this paper is to provide un upper bound of the Gromov width of  homogeneous Hodge manifolds
with second Betti number equal to one.
In this paper a homogeneous Hodge manifold is a compact \K\ manifold $(M, \omega)$ such that 
$\frac{\omega}{\pi}$ is integral and such that  the group of  holomorphic isometries of $M$ acts transitively on $M$.

Our main results are the following three theorems.

\begin{theor}\label{main}
Let $(M, \omega)$ be a compact homogeneous Hodge manifold such that $b_2(M)=1$
and $\omega$ is normalized  so that $\omega(A)=\int_A\omega =\pi$ for the generator  $A\in H_2(M, \z)$. Then 
\begin{equation}\label{cGHSSCT}
c_G (M, \omega)\leq \pi.
\end{equation} 
\end{theor}

\begin{theor}\label{main2}
Let $(M_i, \omega ^i)$, $i=1, \dots, r$, be  homogeneous compact Hodge manifolds as in Theorem \ref{main}.
Then 
\begin{equation}\label{cGprodHSSCT}
c_G \left( M_1\times\dots\times   M_r,  \omega^1\oplus\dots  \oplus \omega^r\right)\leq \pi.
\end{equation}
Moreover, if  $a_1, \dots ,a_r$ are nonzero constants, then
\begin{equation}\label{cGupboundprodHSSCT}
c_G \left( M_1\times\dots\times   M_r, a_1 \omega^1\oplus\dots  \oplus a_r\omega^r\right)\leq \min \{|a_1|, \dots ,|a_r|\}\pi .
\end{equation}
\end{theor}

\begin{theor}\label{main3}
Let $(M, \omega)$ be  as in Theorem \ref{main} and $(N, \Omega)$ be any closed symplectic manifold. Then, for any nonzero real number $a$,
\begin{equation}\label{cGMprodHSSCT} 
c_G(N\times M, \Omega\oplus a\omega )\leq|a|\pi.
\end{equation}
\end{theor}

Note that an Hermitian symmetric space of compact type is an example of compact Hodge manifold with $b_2(M)=1$.
It is worth pointing out that  there exist  many examples of manifolds satisfying the assumption of Theorem 1 which are not symmetric.

In the symmetric case inequalities \eqref{cGHSSCT} and \eqref{cGprodHSSCT}  are  equalities (see \cite{LMZ13} for a proof) and we believe that this holds true also in the non symmetric cases. To this respect recall  a  conjecture due to P. Biran   which asserts that  $\pi$ is a lower bound for the Gromov width of any closed integral symplectic manifold.

The paper contains two other sections. In Section \ref{sectionLu}
we summarize Lu's work on pseudo symplectic capacities and their links with Gromov--Witten invariants needed in the proof of our main results. Section \ref{proofs} is dedicated to the proofs of Theorem 1, 2 and 3.
The paper ends with a remark on  the Seshadri constant of an ample line bundle over a  homogeneous Hodge manifold. 

\section{Pseudo symplectic  capacities}\label{sectionLu}
G. Lu \cite{LU06} defines the concept of \emph{pseudo symplectic capacity} by weakening the  requirements for a symplectic capacity (see the Introduction) in such a way that this new concept depends on the homology classes of the symplectic manifold in question (for more details the reader is referred to \cite{LU06}). More precisely, if one 
denotes by ${{\mathcal C} (2n, k)}$ the set of all tuples $(M, \omega; \alpha_1, \dots, \alpha_k)$ 
consisting of a $2n$-dimensional connected symplectic manifold $(M, \omega)$ and $k$ nonzero homology classes $\alpha_i\in H_*(M; \q )$, $i=1, \dots, k$,
a map $c^{(k)}$ from ${\mathcal C} (2n,k)$ to $[0, +\infty]$
is called a \emph{k-pseudo symplectic capacity}
 if it satisfies the following properties:

({\bf pseudo monotonicity}) if there exists a symplectic embedding $\varphi: (M_1, \omega_1)\rightarrow (M_2, \omega_2)$  then, for any $\alpha_i\in H_*(M_1;  \q)$,
$i=1, \dots ,k$,
\[c^{(k)}(M_1, \omega_1; \alpha_1,\dots ,\alpha_k)\leq c^{(k)}(M_2, \omega_2; \varphi_*(\alpha_1), \dots , \varphi_*(\alpha_k));\]  

({\bf conformality}) $c^{(k)}(M, \lambda\omega; \alpha_1, \dots, \alpha_k)=|\lambda|c^{(k)}(M, \omega; \alpha_1, \dots, \alpha_k)$, for every $\lambda\in\R\setminus \{0\}$
and all homology classes $\alpha_i\in H_*(M; \q)\setminus \{0\}$, $i=1, \dots ,k$;

({\bf nontriviality}) $c(B^{2n}(1), \omega_0; pt, \dots, pt)=\pi =c(Z^{2n}(1), \omega_0; pt, \dots, pt)$,
where $pt$ denotes  the homology class of a point.

\vskip 0.3cm

Note that if $k>1$ a $(k-1)$-pseudo symplectic capacity is defined by
\[c^{(k-1)}(M, \omega; \alpha_1, \dots , \alpha_{k-1}):=c^{(k)}(M, \omega;pt,  \alpha_1, \dots , \alpha_{k-1})\]
and any $c^{(k)}$ induces a true symplectic capacity 
\[c^{(0)}(M, \omega):=c^{(k)}(M, \omega; pt, \dots , pt).\]
Observe also that (unlike  symplectic capacities) pseudo symplectic capacities do not define symplectic invariants. 

\vskip 0.5cm

In   \cite{LU06} G. Lu  was able to construct two  $2$-pseudo symplectic capacities 
denoted by  $C_{HZ}^{(2)}(M, \omega; \alpha_1, \alpha_2)$ and $C_{HZ}^{(2o)}(M, \omega; \alpha_1, \alpha_{2})$ respectively (see Definition 1.3 and  Theorem 1.5 in  \cite{LU06}), where 
$\alpha_1$ and $\alpha_2$ are  homology classes\footnote{In the notations of \cite{LU06}  the generic classes $\alpha_1$  (resp. $\alpha_{2}$) are called $\alpha_0$ (resp. $\alpha_\infty$).

The reason for this notation comes from the concept  of hypersurface $S\subset M$ separating the homology classes $\alpha_0$ and $\alpha_{\infty}$  (see Definition 1.3 and the 
$(\alpha_0, \alpha_{\infty})$-Weinstein conjecture at p.6 of \cite{LU06}).} in $H_*(M;  \q)$.
The $C_{HZ}^{(2)}$ and $C_{HZ}^{(2o)}$ are called by 
Lu   \emph{pseudo symplectic capacities of Hofer--Zehnder type}.
Denote by 
\[C_{HZ}(M, \omega):=C_{HZ}^{(2)}(M, \omega; pt, pt)\]
 (resp. $C^{0}_{HZ}(M, \omega):=C_{HZ}^{(2o)}(M, \omega; pt, pt)$)
the corresponding  true symplectic capacities associated to Lu's pseudo symplectic capacities.
The next lemma summarizes some   properties of the concepts involved so far.
\begin{lem}\label{lemmasumm}
Let $(M, \omega)$ be any symplectic manifold. 
Then, for arbitrary homology classes $ \alpha_1, \alpha_2\in H_*(M; \q)$   and for any nonzero homology class 
$\alpha$, with $\dim \alpha\leq \dim M-1$, the following inequalities hold true:
\begin{equation}\label{basiciC2}
C^{(2)}_{HZ}(M, \omega; \alpha_1, \alpha_2)\leq C^{(2o)}_{HZ}(M, \omega; \alpha_1, \alpha_{2})
\end{equation}
\begin{equation}\label{cG<C2}
c_G(M, \omega)\leq C^{(2)}_{HZ}(M, \omega; pt, \alpha),
\end{equation}
\end{lem}
\begin{proof}
See Lemma 1.4 and (12) in  \cite{LU06}.
\end{proof}

When the symplectic manifold  $M$ is closed  the pseudo symplectic capacities   
$C_{HZ}^{(2)}(M, \omega; \alpha_1, \alpha_2)$ and $C_{HZ}^{(2o)}(M, \omega; \alpha_1, \alpha_{2})$  can be estimated  by other two pseudo symplectic capacities $\GW(M, \omega; \alpha_1, \alpha_2)$ and \linebreak  $\GW_0(M, \omega; \alpha_1, \alpha_2)$.
These $\GW$ and $\GW_0$  are
defined in terms of  \emph{Liu--Tian type Gromov-Witten invariants} as follows. Let $A\in H_2(M, \z)$:
the Liu--Tian  type  Gromov--Witten invariant of genus $g$ and with $k$ marked points  is a homomorphism
\[\Psi^M_{A, g, k}:H_*(\overline{{\mathcal M}}_{g, k}; \q)\times H_*(M; \q)^{k}\rightarrow \q , \ 2g+k \geq 3\]
where $\overline{{\mathcal M}}_{g, k}$ is the space of isomorphism classes of genus $g$ stable curves with $k$ marked points. When there is no risk of confusion, we will omit the superscript $M$ in $\Psi^M_{A, g, k}$.
Roughly speaking, one can think of $\Psi^M_{A, g, k}({\mathcal{C}}; \alpha_1, \dots, \alpha_k)$ as counting, for suitable generic $\omega$-tame almost complex structure $J$ on $M$, the number of $J$-holomorphic curves of genus $g$ representing $A$, with $k$ marked points $p_i$ which pass through cycles $X_i$ representing $\alpha_i$, and such that the image of the curve belongs to a cycle representing ${\mathcal{C}}$ (for details the reader is referred to the Appendix in \cite{LU06} and references therein for details).

\noindent In fact, several different constructions of Gromov-Witten invariants appear in the literature and the question whether they agree is not trivial (see \cite{LU06} and also Chapter 7 in \cite{MCSA94}).  The Gromov--Witten invariants described in the book of D. McDuff and D. Salamon \cite{MCSA94} are the most commonly used:
these  are homomorphisms  
\[\Psi_{A, g, m+2}: H_*(M; \q)^{m+2}\rightarrow \q ,\  m\geq 1\]
which play an important role in the proofs of this paper. The conditions under which these invariants agree with the ones considered by Lu
are given in  Lemma \ref{lemmaugGW} below.

\smallskip

\noindent Let $\alpha_1, \alpha_2\in H_*(M, \q)$.
Following \cite{LU06}, one defines
\[\GW_g (M, \omega; \alpha_1, \alpha_2)\in (0, +\infty]\]
as the infimum of the $\omega$-areas $\omega (A)$ of the homology classes $A\in H_2(M, \z)$ for which the Liu--Tian  Gromov--Witten
invariant \linebreak $\Psi_{A, g, m+2}(C; \alpha_1, \alpha_2, \beta_1, \dots , \beta_m)\neq 0$
for some homology classes $\beta_1, \dots , \beta_m\in H_*(M, \q)$
and $C\in H_*(\overline{{\mathcal M}}_{g, m+2}; \q)$ and integer $m\geq 1$ (we use the convention $\inf \emptyset = + \infty$). The positivity of $GW_g$ reflects the fact that $\Psi_{A, g, m+2} = 0$ if $\omega(A) <0$ (see, for example, Section 7.5 in \cite{MCSA94}).
Set
\begin{equation}\label{GW}
\GW (M, \omega; \alpha_1, \alpha_2):=\inf \{\GW_g(M, \omega; \alpha_1, \alpha_2) \ | \ g\geq 0\}\in [0, +\infty].
\end{equation}

\begin{lem}\label{C2GW}
Let $(M, \omega)$ be a closed symplectic manifold. Then 
\[0\leq\GW (M, \omega;  \alpha_1, \alpha_2)\leq \GW_0 (M, \omega;  \alpha_1, \alpha_2).\]
Moreover $\GW (M, \omega;\alpha_1, \alpha_2)$ and $\GW _0(M, \omega; \alpha_1, \alpha_2)$
are pseudo symplectic capacities and, if $\dim M\geq 4$ then, for nonzero homology classes $\alpha_1, \alpha_2$,
we have
\[C_{HZ}^{(2)}(M ,\omega; \alpha_1, \alpha_2)\leq \GW (M, \omega; \alpha_1, \alpha_2), \]
\[C_{HZ}^{(2o)}(M ,\omega; \alpha_1, \alpha_2)\leq \GW_0 (M, \omega; \alpha_1, \alpha_2).\]
In particular, for every nonzero homology class $\alpha\in H_*(M ,\q)$,
\begin{equation}\label{CHZGW}
C_{HZ}^{(2)}(M ,\omega; pt, \alpha)\leq \GW (M, \omega; pt, \alpha),
\end{equation}
\begin{equation}\label{CHZGW0}
C_{HZ}^{(2o)}(M ,\omega; pt, \alpha)\leq \GW_0 (M, \omega; pt, \alpha).
\end{equation}
\end{lem}
\begin{proof}
See Theorems 1.10 and 1.13 in \cite{LU06}.
\end{proof}

\noindent We end this section with the following lemmata fundamental for the proof of our results.
Recall that a closed symplectic manifold is \emph{monotone} if there exists a number $\lambda >0$
such that $\omega (A)=\lambda c_1(A)$ for $A$ spherical (a homology class is called spherical if it is in the image of the Hurewicz homomorphism $\pi_2(M) \rightarrow H_2(M, \z)$). Further, a homology class $A\in H_2(M, \z)$ is 
 \emph{indecomposable}  if it cannot be decomposed as  a sum $A=A_1+\cdots +A_k$, $k\geq 2$, of classes
which are spherical and satisfy $\omega (A_i)>0$ for $i=1, \dots , k$.
\begin{lem}\label{lemmaugGW} 
Let $(M, \omega)$ be a closed  monotone symplectic  manifold.
Let $A\in H_{2}(M, \z)$ be an indecomposable spherical  class, let $pt$ denote the class of a point in   $H_*(\overline{{\mathcal M}}_{g, m+2}; \q)$
and let   $\alpha_i\in H_*(M, \z)$, $i=1, 2, 3$.
Then the Liu--Tian Gromov--Witten invariant $\Psi_{A, 0, 3}(pt; \alpha_1, \alpha_2, \alpha_3)$ 
agrees with the Gromov--Witten invariant $\Psi_{A,0, 3}(\alpha_1, \alpha_2, \alpha_3)$.
\end{lem}
\begin{proof}
See \cite[Proposition 7.6]{LU06}.
\end{proof}

\begin{lem}\label{prodGromovWitten} 
Let $(N_1, \omega_1)$ and $(N_2, \omega_2)$
be two closed symplectic manifolds.
Then for every integer $k\geq 3$ and homology classes $A_2\in H_2(N_2; \z)$ and $\beta_i\in H_*(N_2; \z)$, $i=1, \dots, k$,
\[\Psi^{N_1\times N_2}_{0\oplus A_2, 0, k}(pt; [N_1]\otimes \beta_1,\dots , [N_1]\otimes \beta_{k-1}, pt\otimes \beta_k)=
\Psi^{N_2}_{A_2, 0, k}(pt; \beta_1, \dots , \beta_k).\]
\end{lem}
\begin{proof}
See  \cite[Proposition 7.4]{LU06}.
\end{proof}

\section{The proofs of Theorems \ref{main},  \ref{main2}, \ref{main3}}\label{proofs}
The following lemma is the  key tool for the   proofs of our main results.

\begin{lem} \label{gromovwittenHSSCT}
Let $(M, \omega)$ be a compact homogeneous Hodge manifold of complex dimension $n$ such that $b_2(M)=1$
and $\omega$ is normalized  so that $\omega(A)=\int_A\omega =\pi$ for the generator  $A\in H_2(M, \z)$. 
Then there exist $\alpha  {(M, \omega)}$ and $\beta (M, \omega )$ in $H_*(M, \z)$ such that
\[\dim \alpha  {(M, \omega )} + \dim\beta  {(M, \omega )}=4n-2c_1(A)\]
and
\begin{equation}\label{nonzeroGW}
\Psi_{A, 0, 3}(pt; \alpha  {(M, \omega )},  \beta  {(M, \omega )}, pt)\neq 0 .
\end{equation}
\end{lem}
\begin{proof}
Since the  symplectic form $\omega$ is K\"ahler-Einstein (being $b_2(M)=1$), it follows that $(M, \omega )$ is monotone, so that Lemma \ref{lemmaugGW}  applies under our assumptions. We need then to show the existence
of a non-vanishing Gromov-Witten invariant $\Psi_{A,0,3}(\alpha
(M, \omega), \beta(M, \omega ), pt)$, which follows by Fulton's results on the {\em quantum cohomology} of homogeneous spaces proved in \cite{FULTON04bis}. In order to explain this, let us recall that a compact homogeneous space $M$ writes as $M= G/P$, where $G$ is a semisimple complex group and $P$ is parabolic (i.e. contains a maximal solvable subgroup of $G$) in $G$. Let $R$ be the {\it root system} associated to $G$. As it is known from the theory of semisimple complex Lie algebras (see, for example, \cite{H} for the details), one chooses in $R$ a finite set $\Delta\subset R$, 
the set of {\it simple roots} (with the property that every root can be written as a linear combination of the elements of $\Delta$ with the coefficients either all non-negative or all non-positive) and associates to every $\alpha \in R$ a reflection $s_{\alpha}$ in a suitable euclidean vector space: the group $W$ generated by these reflections is called the {\it Weyl group} of $G$. For every $w \in W$ the {\it length} $l(w)$ of $w$ is defined as the minimum number of reflections associated to simple roots whose product is $w$. 
Then, as recalled in Section 3 of \cite{FULTON04bis}, the homology classes of $M$ correspond to the classes of the quotient $W/W_P$, where $W_P$ is the subgroup of $W$ generated by the reflections $s_{\alpha}$ for which $\alpha$ belongs to the root system of the reductive part of $P$. More precisely, for every $u \in W/W_P$ we denote by $\sigma(u)$ (resp. $\sigma_u$) the class of the corresponding so called {\it Schubert variety} (resp. {\it opposite Schubert variety} ) of complex dimension (resp. codimension) $l(u):= \inf_{[w] = u} l(w)$. The classes $\sigma_u$ and $\sigma(u)$ are dual under the intersection pairing, and moreover one has $\sigma(u) = \sigma_{u^{\vee}}$, where $u^{\vee} := w_0 u$, being $w_0$ the element of longest length in $W$. In particular, $\sigma_{[w_0]} = \sigma_{[1]^{\vee}} = \sigma([1])$ has zero dimension, i.e. is the class of a point.

 Under the assumption $b_2(G/P) = 1$, there exists a simple root $\beta$ such that $H_2(G/P)$ is generated by the class $\sigma([s_{\beta}])$. Now, for every $u = [\tilde u] \in W/W_P$, let $v = [\tilde u s_{\beta}]^{\vee}$. Then, in the terminology of Section 4 in \cite{FULTON04bis}, $u$ and $v^{\vee}$ are {\it adjacent} and $u_0 = u, u_1 = v^{\vee}$ is a {\it chain of degree $\sigma([s_{\beta}])$} between $u$ and $v$. By Theorem 9.1 in \cite{FULTON04bis}, it follows that there exists $w \in W/W_P$ such that the Gromov-Witten invariant  $\Psi_{\sigma([s_{\beta}]),0,3}(\sigma_u, \sigma_v, \sigma_{w^{\vee}})$  does not vanish. In particular, when $u = [w_0]$, we get $\Psi_{\sigma([s_{\beta}]),0,3}(pt, \sigma_v, \sigma_{w^{\vee}}) \neq 0$, as required. The relation between the (real) dimensions of the classes in the statement easily follows from the general condition necessary for a Gromov-Witten invariant to be well-defined and non-vanishing (see, for example, Section 6 in \cite{FULTON04bis}).
\end{proof}
\begin{proof}[Proof of Theorem \ref{main}]
In order to use  Lemma \ref{C2GW} we can assume, without loss of generality, that  $\dim M\geq 4$.
Indeed  the only compact homogeneous Hodge manifold of (real)  dimension $< 4$ (and hence of dimension $1$)
is  $(\CP^1, \omega_{FS})$ 
whose  Gromov width is well-known to be equal to $\pi$.
Let $A=[\C P^1]$ be the generator of $H_2(M, \z)$ as in the statement of Theorem \ref{main}. Then the value $\omega (A)=\pi$
is clearly the infimum of the $\omega $-areas $\omega  (B)$ of the homology classes $B\in H_2(M, \z)$ for which $\omega (B) > 0$.
By Lemma \ref{gromovwittenHSSCT} we have $\Psi_{A, 0, 3}(pt; pt, \alpha , \beta ) \neq 0,$
with $\alpha=\alpha (M, \omega )$ and  $\beta=\beta (M, \omega )$, and hence, by definition of $GW_g$,
\begin{equation}\label{GW=GW0}
\GW (M, \omega ; pt, \gamma)=\GW_0 (M, \omega ; pt, \gamma) =\pi
\end{equation}
with $\gamma=\alpha (M, \omega )$ or  $\gamma=\beta (M, \omega)$.
It follows by the  inequalities
(\ref{basiciC2}), (\ref{cG<C2}), (\ref{CHZGW}) and   (\ref{CHZGW0}) that 
$$c_G(M, \omega )\leq C_{HZ}^{(2)}(M ,\omega ; pt, \gamma) \leq
C_{HZ}^{(2o)}(M ,\omega ; pt, \gamma)\leq\pi ,$$
i.e.  the desired inequality. 
\end{proof}

\noindent
{\em Proof of Theorem \ref{main2}.} The proof of  Theorem \ref{main2} is an immediate consequence of the following
result combined with (\ref{basiciC2}) and  (\ref{cG<C2}) in Lemma \ref{lemmasumm}.

\begin{lem}
Let $(M, \omega)$ be as in Theorem \ref{main} and let $(N, \Omega)$ be any closed symplectic manifold.
Then 
\begin{equation}\label{C2oNM}
C_{HZ}^{(2o)}(N\times M, \Omega\oplus a\omega; pt, [N]\times\gamma)\leq |a|\pi
\end{equation}
for any $a\in\R\setminus\{0\}$ and $\gamma=\alpha  {(M, \omega )}$ or $\gamma=\beta  {(M, \omega )}$, 
with $\alpha  {(M, \omega)}$ and $\beta  {(M, \omega )}$ given by Lemma \ref{gromovwittenHSSCT}.
\end{lem}
\begin{proof}
Since by   (\ref{nonzeroGW}) we have $\Psi^M_{A, 0, 3}(pt; \alpha,  \beta , pt) \neq 0,$
with $\alpha =\alpha (M, \omega )$ and  $\beta =\beta (M, \omega )$,
it follows by Lemma  \ref{prodGromovWitten}  that
\[\Psi^{N \times M}_{B, 0, 3}(pt; [N]\times \alpha (M, \omega), [N]\times \beta (M, \omega), pt )\neq 0
\]
for $B=0\times A$, where $0$ denotes the zero class in $H_2(N, \z)$ and $A$ the generator of $H_2(M, \z)$.
Hence (\ref{C2oNM})  easily follows   from (\ref{CHZGW0}) in Lemma \ref{C2GW}.
\end{proof}

\begin{proof}[Proof of Theorem \ref{main3}]
From  (\ref{basiciC2}) and  (\ref{cG<C2}) in Lemma \ref{lemmasumm}  and by (\ref{C2oNM}) it follows  that 
\[c_G(N\times M, \Omega\oplus a\omega )\leq  C_{HZ}^{(2o)}(N\times M, \Omega\oplus a\omega ; pt, [N]\times\gamma)\leq |a|\pi,\]
where $\gamma=\alpha  {(M, \omega )}$ (or $\gamma=\beta  {(M, \omega )}$),
which yields the desired inequality (\ref{cGMprodHSSCT}).
\end{proof}

\noindent
{\bf On Seshadri constants of homogeneous manifolds}

\noindent
Given a compact complex manifold $(N, J)$
and a holomorphic line bundle  $L\rightarrow N$
the \emph{Seshadri constant} of $L$ at a point $x\in N$ is defined as the nonnegative real number 
\[\epsilon (L, x)=\inf_{C\ni x}\frac{\int_Cc_1(L)}{\mult_xC}, \]
where the infimum is taken over all irreducible holomorphic curves $C$ passing through the point $x$
and $\mult_xC$ is the multiplicity of $C$ at $x$ (see \cite{DE} for details). 
The (global) Seshadri constant is defined by
\[\epsilon (L)=\inf_{x\in N}\epsilon (L, x).\]
Note that Seshadri's criterion for ampleness says that  $L$ is ample if and only if  $\epsilon (L)>0$.
P. Biran and K. Cieliebak \cite[Prop. 6.2.1]{BICI01} have  shown  that  
\[\epsilon (L)\leq c_G(N, \omega_L),\]
where $\omega_L$ is any  \K\ form 
which represents the first Chern class of $L$, i.e. $c_1(L)=[\omega_L]$.
Let now $(M, \omega )$
be a compact homogeneous Hodge manifold and $L$
be the (very) ample  
line bundle  $L\rightarrow M$  such that 
$c_1(L)=[\frac{\omega}{\pi}]$
($L$ can be taken as the pull-back of a suitable normalized Kodaira embedding $M\rightarrow \C P^N$  of the universal bundle of $\CP^N$).
Then, by using the upper bound $c_G(M, \omega )\leq\pi$  and   the conformality of $c_G$ 
we get: 
\begin{cor}\label{maincor}
Let $(M_i, \omega ^i)$, $i=1, \dots, r$, be  homogeneous compact Hodge manifolds as in Theorem \ref{main}.
Let $(M, \omega)= (M_1\times\dots\times   M_r,  \omega^1\oplus\dots  \oplus \omega^r)$ and $L\rightarrow M$ as above.
Then $\epsilon (L)\leq 1.$
\end{cor}

\begin{remark}
The previous inequality has been found in \cite{LMZ13} when $(M, \omega)$ is a Hermitian symmetric space of compact type 
\end{remark}

\section*{Acknowledgments}
The first two  authors were  supported by Prin 2010/11 -- Variet\`a reali e complesse: geometria, topologia e analisi armonica -- Italy. The third author was supported by Firb project 2012 -- Geometria Differenziale e Teoria geometrica delle funzioni -- Italy.

\end{document}